\documentclass[12pt]{article}
\usepackage{latexsym}
\usepackage{a4wide}

\newtheorem{theorem}{Theorem}
\newtheorem{lemma}{Lemma}[section]
\newtheorem{corollary}[theorem]{Corollary}

\newenvironment{proof}{\rmfamily\upshape\mdseries\small{\noindent\normalsize\bfseries Proof:}}{\nopagebreak\rule{2mm}{2mm}\newline}

\def\s{\,\,}

\begin{document}
\title{Almost all permutations power to a prime length cycle}
\author{W. R. Unger\thanks{The author acknowledges the support of Australian Research Council grant DP160104626} \\
School of Mathematics and Statistics, \\
University of Sydney, \\
Sydney, Australia}
\date{\today}

\maketitle

\begin{abstract}
We show that almost all permutations have some power that is a cycle of prime
length.
The proof includes a theorem giving a strong upper bound on the proportion of
elements of the symmetric group having no cycles with length in a given set.
\end{abstract}

\section{Introduction}
In this note we prove that almost all permutations
have some power that is a cycle of prime length.
Equivalently, that almost all permutations have a cycle
of prime length, $p$ say, where $p$ does not divide the length of any other
cycle of the permutation.

This question arises from C. Jordan's results of the 1870s, where he showed
that if $G$ is a primitive permutation group of degree $n$, and $G$ contains a
$p$-cycle, where $p\le n-3$ and $p$ is prime, then $G$ contains the alternating
group of degree $n$. Jordan's result has been generalised several times,
see for instance \cite{J14} for a generalisation to non-prime cycles.

When computing with permutation groups such results are used as a test for
a primitive permutation group to contain the alternating group, and so merit
special treatment. It is of interest to know that as the degree grows,
it becomes easier to find such elements 
by random search in the giant permutation groups $A_n$ and $S_n$.

Thee next section gives terminology and notation. Following that we give a
simple but quite strong upper bound on the proportion of permutations in
$S_n$ having no cycle lengths in some fixed set.
The last section proves the result of the title, and some easy corollaries.

\section{Terminology and Notation}
We use $S_n$ to denote the symmetric group on the set $\{1,2\dots,n\}$,
and $A_n$ to denote the alternating subgroup of $S_n$.

When $P$ is some property of permutations we say that $P$ holds for
\emph{almost all permutations} when $\lim_{n\to\infty} p_n = 1$,
where $p_n$ is the proportion of elements of $S_n$ having property $P$.
When considering properties of a permutation $\sigma$, we will use $n$ to
denote the degree of the containing group $S_n$.

When we refer to the cycles of a permutation, we refer to the factors in the
product of disjoint cycles decomposition of the permutation. We will say
that a permutation is a \emph{cycle}
when the cycles of the permutation consist of one cycle of
length greater than 1, and some number (perhaps zero) of fixed points.
When the non-trivial cycle has length $\ell$, the permutation is called an
$\ell$-cycle.

A summation with index $p$ indicates a sum over all primes $p$ in the
given range.

We use the Bachmann-Landau symbols $O()$ and $o()$ (big-O and little-o)
as defined in \cite[\S1.6, p7]{HW}.

All logarithms here are to base $e$.

\section{An upper bound}\label{upper}
In this section we fix a degree $n$ and a set $C$ of possible cycle lengths,
and consider the proportion of elements of $S_n$ having no cycles with length
in $C$. We will prove an upper bound on this proportion.

The parameter $\mu = \sum_{k\in C} 1/k$ is important here.
An upper bound of $1/\mu$ for the proportion was given in \cite[Theorem~VI]{ET},
while \cite[p39]{Mans}, gives an upper bound of $\exp(\gamma-\mu)(1+1/n)$. 
The following result is a very small improvement on this last.
The article referred to in \cite{Mans} was inaccessible to the current author,
which led to the proof below, and the simplicity of the bound is striking.
\begin{theorem}\label{mu_bnd}
Let $n$ be a positive integer, let $C$ be any subset of $\{1,2,\dots,n\}$,
and put $\mu = \sum_{k\in C} 1/k$.
Then the proportion of elements of $S_n$ having no cycle with length in
$C$ is less than $\exp(\gamma-\mu)$, where $\gamma$ is Euler's constant.
\end{theorem}
\begin{proof}
Define, for $k = 1,2,\dots n$, $a_k = 0$ for $k\in C$, with $a_k=1$ otherwise.
Set 
$$A(z) = \sum_{k=1}^n \frac{a_k}{k} z^k,\qquad\mbox{and}
\qquad F(z) = \exp(A(z)).$$
This $F(z)$ is an entire function and we take the Taylor series for
$F(z)$ about $0$ to be
$$
F(z) = \sum_{k=0}^\infty p_k z^k.
$$
We can identify $p_k$, for all $k\ge0$, as the proportion of elements of
$S_k$ having all cycles with length $\le n$, and having no cycles with length
in $C$.
This may be seen by \cite[Proposition II.4]{FS09}, or by applying the
methods of \cite[Chapter~3]{Wilf}.
With this notation we are trying to prove that $p_n < \exp(\gamma-\mu)$.

Define $E(n)$, for $n\ge1$, by
$$ \sum_{k=1}^n \frac{1}{k} \s=\s \log n + \gamma + E(n).$$
It is well known that $0 < E(n) < 1/2n$, and so $\exp(E(n))<1+1/n$.
Now
$$
A(1) \s=\s \sum_{k=1}^n \frac{1}{k} - \mu \s=\s \log n + \gamma-\mu + E(n).
$$
Using these relations we get
$$
F(1) \s<\s n\exp(\gamma-\mu)\left(1+\frac{1}{n}\right) \s=\s
(n+1)\exp(\gamma-\mu).
$$
Now consider the derivative of $F(z)$, 
\begin{equation}\label{deriv}
F'(z) \s=\s A'(z)F(z).
\end{equation}
Observe that $A'(z) = \sum_{k=1}^n a_k z^{k-1}$ and equate
the coefficients of $z^{n-1}$ on either side of (\ref{deriv}):
$$
np_n \s=\s \sum_{k=0}^{n-1} a_{n-k}p_k \s\le\s
\sum_{k=0}^{n-1} p_k \s\le\s F(1) - p_n.
$$
It follows that
$$
(n+1)p_n \s\le\s  F(1) \s<\s (n+1)\exp(\gamma-\mu),
$$
which proves the result.
\end{proof}

If $\mu > 0$ then $\exp(\mu-\gamma) > \exp(\mu-1) \ge \mu$, so
$\exp(\gamma-\mu) < 1/\mu$.
This shows that the bound of Theorem \ref{mu_bnd} is always stronger than
the bound of \cite[Theorem~VI]{ET}.

Consider the sequence of examples $C_n = \{1,2,\dots n-1\}$.
The elements of $S_n$ with no cycles having length in $C_n$ are the $n$-cycles.
As $n\to\infty$, the upper bound of Theorem \ref{mu_bnd} is asymptotic
to $1/n$, which is the exact proportion of $n$-cycles in $S_n$.
For these examples the upper bound looks very good.

When $C$ is a very small set, \cite[Theorem B]{Mans} implies that
the proportion of permutations with no cycle lengths in $C$
will be close to $e^{-\mu}$.
The extreme examples here are $C=\{\ell\}$ for some fixed $\ell$.
It is known that the proportion of elements without an
$\ell$-cycle is asymptotic to $e^{-1/\ell}$ as $n\to\infty$,
while our upper bound is $e^{\gamma-1/\ell}$, so the $e^\gamma=1.781\dots$
multiplier looks larger than necessary.
Indeed the upper bound given above is useless (that is $>1$) for $\ell\ge2$.

In the next section we will make use of sets $C$ where 
$n/|C| \approx \log n$ and $e^{-\mu}\approx 2\log\log n/\log n$,
a situation between the extreme examples above.

\section{The Main Theorem}
We will now prove the title statement.
We first show that almost all permutations have a cycle of prime length
where the prime is not too small.
As we are looking at properties of $S_n$ for large $n$, we will assume
that $n > 10$.

We will use an auxiliary function $f(n) = (\log n)^2$.
A different $f$ could be used.
For the proof of the next lemma, $f$ must not grow too quickly as $n\to\infty$
($\log f(n) = o(\log n)$), and, for the proof of the Main Theorem,
$f$ must not grow too slowly ($\log n = o(f(n))$).

\begin{lemma}\label{prime_cyc}
For almost all permutations $\sigma$, there exists a prime $p$
with $p > f(n)$, such that $\sigma$ has a cycle of length $p$.
\end{lemma}
\begin{proof}
A theorem of Mertens \cite[Theorem~427]{HW} tells us that there exists a
constant $M$ so that
\begin{equation}\label{mertens}
\sum_{p\le x}\frac{1}{p} \s=\s \log\log x + M + o(1)
\end{equation}
as $x\to\infty$. Now considering primes beteween $f(n)$ and $n$ we have
\begin{eqnarray}
\sum_{f(n)<p\le n} \frac{1}{p}
    & = & \log\log n - \log\log f(n) + o(1) \nonumber\\
    & = & \log\log n - \log\log\log n - \log 2 + o(1). \label{loglog} 
\end{eqnarray}
The result follows from Theorem~\ref{mu_bnd} with $C$ equal
to the set of primes in the interval $(f(n), n\,]$, which shows that the
proportion of elements in $S_n$ not having a cycle of the required prime
length is $O(\log\log n/\log n)$, so goes to zero as $n\to\infty$.
\end{proof}

The proof of the following theorem shows that for almost all permutations, a 
sufficiently long prime length cycle, as in Lemma~\ref{prime_cyc}, is in
general the only cycle with length a multiple of the prime.
\begin{theorem}
\label{main}
For almost all permutations $\sigma$, there is some power of $\sigma$ that
is a cycle of prime length.
\end{theorem}
\begin{proof}
We continue to use $f(n) = (\log n)^2$, and $n>10$ so $f(n) > 5$.

Let $T_n \subseteq S_n$ be the set of permutations $\sigma$ such that there
exists a prime $p$, with $p > f(n)$, such that $\sigma$ has a cycle of
length $p$.  Lemma~\ref{prime_cyc} implies
\begin{equation}\label{t_lim}
\lim_{n\to\infty} \frac{\left|T_n\right|}{n!} = 1.
\end{equation}

Fix a prime $p > f(n)$, and consider $U_{n,p}\subseteq S_n$
consisting of the permutations that have a cycle of length $p$, and have
another cycle of length divisible by $p$.
Then $U_{n,p}\subseteq T_n$, and note that $p>n/2$ implies $U_{n,p}$ is
empty.

We find an upper bound for $\left|U_{n,p}\right|$ by
counting all structures consisting of a permutation from $S_n$ having a first
distinguished cycle of length $p$, and a second distinguished cycle with
length divisible by $p$.
This will be an upper bound for $\left|U_{n,p}\right|$ as every element of
$U_{n,p}$ gives rise to such a structure, some giving more than one structure
by distinguishing different cycles.
In the following, the right hand side of (\ref{n_structs}) counts the
structures mentioned, where the length of the second distinguished cycle
is $kp$, and $p>5$ so $\log p > 1$.
\begin{eqnarray}
\left|U_{n,p}\right| & \le &
    (p-1)!{n \choose p} \sum_{k=1}^{\lfloor n/p\rfloor-1}
	{n-p \choose kp}(kp-1)!\,(n-p-kp)! \label{n_structs}\\
 & = & \frac{n!}{p^2} \sum_{k=1}^{\lfloor n/p\rfloor-1} \frac{1}{k}.\nonumber\\
\frac{\left|U_{n,p}\right|}{n!} & \le &
 \frac{1}{p^2} \sum_{k=1}^{\lfloor n/p\rfloor-1} \frac{1}{k} 
 \s<\s \frac{1}{p^2}\left(1 + \log\frac{n}{p}\right) \s<\s
      \frac{\log n}{p^2}.\label{Unp_bnd}
\end{eqnarray}

Now we bound the following sum over primes greater than $x>1$:
\begin{equation} \label{p2_bnd}
\sum_{p>x}\frac{1}{p^2}\s<\s\sum_{k=\lceil x\rceil}^\infty\frac{1}{k^2}%
\s<\s\frac{1}{x^2}+\int_x^\infty\frac{dt}{t^2}%
\s=\s\frac{1}{x^2}+\frac{1}{x}\s<\s\frac{2}{x}.
\end{equation}
It can be shown that $\sum_{p>x}1/p^2$ is asymptotic to $1/(x\log x)$
as $x\to\infty$, so this bound is far from best possible,
but it is easy to derive and is enough for what follows.

Let $U_n = \bigcup_{p > f(n)} U_{n,p}$. 
We have $U_n\subseteq T_n$ and, using (\ref{Unp_bnd}) and (\ref{p2_bnd}),
\begin{equation} \label{u_bnd}
\frac{\left|U_n\right|}{n!} \s\le\s
	\sum_{p>f(n)} \frac{\left|U_{n,p}\right|}{n!}
 \s<\s \sum_{p>f(n)} \frac{\log n}{p^2} 
 \s<\s \frac{2\log n}{f(n)}
 \s=\s \frac{2}{\log n}.
\end{equation}
We deduce from (\ref{u_bnd}) that
\begin{equation}\label{u_lim}
\lim_{n\to\infty} \frac{\left|U_n\right|}{n!} = 0.
\end{equation}

Put $V_n = T_n \setminus U_n$. Let $\sigma\in T_n$, so $\sigma$ has a cycle of
prime length $p$ with $p>f(n)$. If no power of $\sigma$ is a $p$-cycle
then $\sigma$ has some other cycle with length divisible by $p$, so
$\sigma\in U_n$. It follows that every element of $V_n$ has the property
that some power is a cycle of prime length.
Furthermore, from (\ref{t_lim}) and (\ref{u_lim}),
$$
\lim_{n\to\infty} \frac{|V_n|}{n!} =
\lim_{n\to\infty} \frac{|T_n|-|U_n|}{n!} = 1.
$$
\end{proof}

Relating this to C. Jordan's result, we have:
\begin{corollary}\label{cor1}
Almost all permutations have a power that is a $p$-cycle with
$p$ prime and $p\le n-3$.
\end{corollary}
\begin{proof}
The proportion of elements in $S_n$ having a cycle of length $n$, $n-1$,
or $n-2$ is $O(1/n)$.
So almost all permutations have no cycle of length $> n-3$, thus almost
all permutations power to a prime length cycle with cycle length $\le n-3$.
\end{proof}

The following shows that the results above also hold for almost all even
permutations, so apply to algorithms recognising the alternating group as
well as the symmetric group.
\begin{corollary}\label{cor2}
Let $q_n$ be the proportion of elements of the alternating group $A_n$ that
have a power that is a $p$-cycle with $p$ prime and $p\le n-3$. Then
$\lim_{n\to\infty} q_n = 1$.
\end{corollary}
\begin{proof}
Let $r_n$ be the proportion of elements of $S_n$ that have the given property.
By Corollary~\ref{cor1}, $\lim_{n\to\infty} r_n = 1$. As $A_n$ has index 2 
in $S_n$, we have $0\le 1-q_n \le 2(1-r_n)$, and the result follows by taking
the limit $n\to\infty$.
\end{proof}

Finally, using the results above we may prove a sharper version
of the main theorem.
\begin{theorem}
The proportion of elements in $S_n$ that do not power to a cycle of prime
length is at most $O(\log\log n/\log n)$.
\end{theorem}
\begin{proof}
We use notation from the proof of Theorem~\ref{main}. It follows from
Theorem~\ref{mu_bnd} and equation~(\ref{loglog}) that 
$$\frac{|S_n\setminus T_n|}{n!} = O\left(\frac{\log\log n}{\log n}\right).$$
The proof of Theorem~\ref{main} gives
$$\frac{|U_n|}{n!} = O\left(\frac{1}{\log n}\right) =
o\left(\frac{\log\log n}{\log n}\right).$$
Thus we find
$$\frac{|S_n\setminus V_n|}{n!} = O\left(\frac{\log\log n}{\log n}\right).$$
Since all elements of $V_n$ power to a prime cycle this proves the result.
\end{proof}

%
%

\bibliographystyle{plain}
\bibliography{is_altsym}

\end{document}